\documentclass{amsart}

\usepackage{amsmath,amssymb,amsthm}

\usepackage[latin1]{inputenc}
\usepackage[T1]{fontenc}
\usepackage[all]{xy}
\usepackage{enumerate}

\newtheorem{thm}{Theorem}[section]
 \newtheorem{cor}[thm]{Corollary}
 \newtheorem{lem}[thm]{Lemma}
 \newtheorem{prop}[thm]{Proposition}

 \theoremstyle{definition}
 \newtheorem{defn}[thm]{Definition}
 \newtheorem{notation}[thm]{Notation}
 \newtheorem*{df*}{Definition}
 \newtheorem{algo}[thm]{Algorithm}

 \theoremstyle{remark}
 \newtheorem{rem}[thm]{Remark}
 \newtheorem*{ex}{Example}
 \numberwithin{equation}{section}

 \newcommand{\K}{\mathbb{K}}
\newcommand{\N}{\mathbb{N}}
\newcommand{\Hom}{\operatorname{Hom}}
\newcommand{\cC}{\mathcal{C}}
\newcommand{\supp}{\operatorname{supp}}
\newcommand{\FP}{\operatorname{FP}}

\def\norm#1{\lvert\lvert #1\rvert\rvert}

\title[An algorithm for producing F-pure ideals]{An algorithm for producing F-pure ideals}

\author[A.\,F.\,Boix]{Alberto F.\,Boix$^{*}$}
\thanks{$^{*}$Partially supported by MTM2010-20279-C02-01.}
\address{Department of Economics and Business, Universitat Pompeu Fabra, Jaume I Building, Ramon Trias Fargas 25-27, 08005 Barcelona, Spain.}
\email{alberto.fernandezb@upf.edu}
\urladdr{http://atlas.mat.ub.edu/personals/aboix/}

\author[M.\,Katzman]{Mordechai Katzman$^{**}$}
\thanks{$^{**}$Supported by EPSRC grant EP/I031405/1.}
\address{Department of Pure Mathematics, University of Sheffield, Hicks Building, Sheffield S3 7RH, United Kingdom}
\email{M.Katzman@sheffield.ac.uk}
\urladdr{http://www.katzman.staff.shef.ac.uk/}



\begin{document}

\begin{abstract}
This paper describes a method for computing all $F$-pure ideals for a given Cartier map of a polynomial ring over a finite field.
\end{abstract}

\maketitle

\section*{Introduction}

The subject of this paper is the study of certain ideals associated with a given $p^{-e}$-linear map. These maps were introduced by K.\,Schwede in \cite{Schwede2011} and M.\,Blickle in \cite{Blickle2013} in the context of test ideals and are defined as follows.

Throughout this manuscript, unless otherwise is specified, we shall denote by $A$ a fixed regular ring containing $\K$, where $\K$ is an \emph{$F$-finite field} of prime characteristic $p$ (i.\,e. a field $\K$ which is a finite extension of $\K^p$). The \emph{Frobenius map}, raising an element $a\in A$ to its $p$th power $a^p$,
is an additive map.


Given any $A$-module $M$ and $e\in\N$, $F_*^e M$ will denote the abelian group $M$ with $A$-module structure given by $a \cdot m = a^{p^e} m$ for any $a\in A$ and $m\in F_*^e M$. Given an $m\in M$ we shall henceforth write $F_*^e m$ for the same element regarded as a member of $F_*^e M$.

The $p^{-e}$-linear maps referred to above are elements in $\Hom_A (F_*^e M, M)$; these can be thought as additive maps $\xymatrix@1{M\ar[r]^-{\phi_e}& M}$ for which, for any
$a\in A$ and $m\in M$, $\phi_e (a^{p^e}m)=a\phi_e (m)$. We further define
\[
\cC^M :=\bigoplus_{e\geq 0} \Hom_A (F_*^e M, M)
\]
and endow it with the structure of an $A$-algebra by defining the product of $\phi_e\in\cC_e^M$ and $\phi_{e'}\in\cC_{e'}^M$
as the element of $\cC_{e+e'}^M$ given by
\[
\xymatrix{F_*^{e+e'} M\ar[rr]^-{\phi_{e'}\circ F_*^{e'} \phi_e}& & M.}
\]
Moreover, we also set
\[
\cC_+^M :=\bigoplus_{e\geq 1}\Hom_A (F_*^e M, M).
\]
In this article we shall be interested mostly in $\cC^A$.

\begin{df*} 
A $\cC^M$-submodule $N$ of $M$ is \emph{$F$-pure} if $\cC_+^M N=N$. In particular, we say that an ideal $I$ of $A$ is \emph{$F$-pure} provided $\cC_+^A I=I$.
\end{df*}

Our paper is motivated by the study of $F$-pure ideals and their properties introduced in \cite{Blickle2013}.
When $A$ is $F$-finite, $A$ itself is an $F$-pure ideal if and only if $\cC_+^A$ contains a splitting of a certain power of the Frobenius map on $A$ (cf.\,\cite[Proposition 3.5]{Blickle2013}); therefore, these $F$-pure ideals turn out to be a generalization of
the $F$-purity property.

Furthermore, among the main results in \cite{Blickle2013} (see also \cite[Corollary 4.20 and Proposition 5.4]{BlickleBockle2011}) is the fact that the set of $F$-pure ideals in $A$ is finite, and that the big test ideal is the minimal element of the set of $F$-pure ideals.
Apart from their usefulness in describing big test ideals, we believe that the set of  $F$-pure ideals provides an interesting set of invariants of $A$ providing information about the ring which is not yet fully understood.

One should contrast this with the situation one encounters when studying the set of \emph{$\cC_+^A$-compatible ideals},
i.e., ideals $I\subseteq A$ for which $\cC_+^A I \subseteq I$.
One might hope to list all $F$-pure ideals by listing all compatible ideals and checking which ones are $F$-pure.
However, the set of compatible ideals need not be finite, and one can only describe algorithmically the radical ideals among these; this task was carried out in \cite{KatzmanSchwede2012}.

Our contribution to the understanding of $F$-pure ideals is to provide an effective procedure to calculate all the $F$-pure ideals of $A=\K [x_1,\ldots ,x_d]$ contained in the maximal ideal $\mathfrak{m}=\langle x_1,\ldots ,x_d\rangle$ of the subalgebra $\cC=\cC^{\phi}$ of $\cC^A$ generated by one homogeneous element $\phi\in\Hom_A (F_*^e A,A)$ under the additional assumption that the ground field $\K$ is \emph{finite}. This procedure has been implemented in Macaulay2 (cf.\,\cite{BoixKatzmanM2}).

This paper is organized as follows.
Firstly, in Section \ref{theoretical section} we introduce compatible and fixed ideals; moreover, we show that the so-called \emph{$e$th root ideal} (cf.\,Definition \ref{the root ideal}) plays a key role in their calculation (cf.\,Theorem \ref{compatibles y fijos: como calcularlos}). Secondly, Section \ref{the section of the algorithm} contains the main result of this paper; namely, the algorithm referred to above (cf.\,Theorem \ref{Moty nos da un algoritmo}). 
This introduces a new operation on ideals (cf.\,Definition \ref{la construccion hash}), hoping that it may be interesting in its own right. Finally, in Section \ref{examples section} we provide examples in order to illustrate how our method works; most of these specific computations were carried out with an implementation of this procedure in Macaulay2.

\section{Ideals compatible and fixed under a given $p^{-e}$-linear map}\label{theoretical section}

Unless otherwise is specified, $\K$ denotes an \emph{$F$-finite field} of prime characteristic $p$, i.e.,
a field $\K$ which is a finite extension of $\K^p$. Given an $\mathbf{\alpha}=(a_1,\ldots ,a_d)\in\N^d$ we shall use the following multi-index notation:
\[
\mathbf{x}^{\mathbf{\alpha}}:=x_1^{a_1}\cdots x_d^{a_d}.
\]
Moreover, in this case, we set $\norm{\mathbf{x}^{\mathbf{\alpha}}}:=\max\{a_1,\ldots ,a_d\}$ and, for any polynomial $g\in\K [x_1,\ldots ,x_d]$,
\[
\norm{g}:=\max_{\mathbf{\alpha}\in\supp (g)} \norm{\mathbf{x}^{\mathbf{\alpha}}},
\]
where $g=\sum_{\mathbf{\alpha}\in\N^d} g_{\mathbf{\alpha}} \mathbf{x}^{\mathbf{\alpha}}$ (such that $g_{\mathbf{\alpha}}=0$ up to a finite number of terms) and
\[
\supp (g):=\left\{\mathbf{\alpha}\in\N^d\mid\quad g_{\mathbf{\alpha}}\neq 0\right\}.
\]
Given any ideal $I$, $I^{[p^e]}$ will denote the ideal generated by all the $p^e$ powers of elements in $I$.
It is straightforward to verify that $I^{[p^e]}$ is generated by the $p^e$ powers of a set of generators of $I$. Finally, given another ideal $J$ of $A$,
\[
(I:_A J):=\{a\in A\mid\ aJ\subseteq I\}
\]
will denote the corresponding colon ideal; in case $J$ is generated by a single element (namely, $u$), we shall simply write $(I:_A u)$.

The $F$-finiteness of $\K$ implies that $F_*^e A$ is a free $A$-module of finite rank.
Indeed, if $\mathcal{B}_e$ is a $\mathbb{K}^{p^e}$-basis for $\mathbb{K}$, then $F_*^e A$ has free basis
\[
\{b\mathbf{x}^{\mathbf{\alpha}}\mid\quad b\in\mathcal{B}_e,\quad 0\leq\norm{\mathbf{\alpha}}\leq p^e-1\}.
\]
Moreover, we recall that the \emph{trace map} $\Phi_e\in\Hom_A (F_*^e A,A)$, which is the projection onto the direct summand
$A x_1^{p^e-1} \cdots  x_d^{p^e-1}$,
generates the $F_*^e A$-module $\Hom_A (F_*^e A,A)$
(cf.\,\cite[Example 1.3.1]{BrionKumar2005}).
In this way, any homogeneous element $\phi\in\Hom_A (F_*^e A,A)$ can be written as $u\Phi_e$
(interpreted as the composition of multiplication by $u$ followed by $\Phi_e$) for some $u\in F_*^e A$.
Now if $\cC$ is the Cartier subalgebra of $\cC^A$ generated by such a $\phi$ then the problem of finding the $F$-pure ideals of $A$ amounts to finding all ideals $I\subseteq A$ such that $\phi\left(F_*^e I\right)=I$.

\begin{defn}
Let $I$ be an ideal of $A$ and let $\phi\in\Hom_A (F_*^e A,A)$.
\begin{enumerate}[(i)]

\item We say that $I$ is \emph{$\phi$-compatible} if $\phi (F_*^e I)\subseteq I$.

\item We say that $I$ is \emph{$\phi$-fixed} if $\phi (F_*^e I)= I$.

\end{enumerate}
\end{defn}

Clearly, all $\phi$-fixed ideals are $\phi$-compatible. The converse also holds if $\phi$ is a \emph{Frobenius splitting}, i.e., if $\phi(F_*^e 1)=1$: in this case, for any $r$ in a $\phi$-compatible ideal  $I$ we have $\phi(F_*^e r^{p^e})=r \phi( F_*^e 1)= r$.

From now on, we shall write  our given $\phi\in\Hom_A (F_*^e A,A)$ as $u\Phi_e$, where $u\in F_*^e A$.

\subsection{The ideal of $p^e$-th roots}
Our next goal is
to express in an equivalent way the condition of being $\phi$-fixed in order to perform explicit calculations.
Such an equivalent expression requires us to review the following concept (cf.\,\cite[Definition 2.2]{BlickleMustataSmith2008} and \cite[Section 5]{Katzman2008}).

\begin{defn}\label{the root ideal}
Let $J$ be an ideal of $A$.
We set $I_e (J)$ as the smallest ideal $I$ such that $I^{[p^e]}\supseteq J$.
We shall refer to $I_e (J)$ as the \emph{$e$-th root ideal} of $J$ (it is sometimes denoted $J^{[1/p^e]}$).
\end{defn}

We have the following elementary
properties of $e$-th roots (see either \cite[Section 5]{Katzman2008} or \cite[Lemma 2.4 and Proposition 2.5]{BlickleMustataSmith2008} for details).

\begin{prop}\label{propiedades del ideal raiz}
Let $J, J_1,\ldots ,J_r$ be ideals of $A$. Then, the following statements hold.
\begin{enumerate}[(a)]


\item If $J_1\subseteq J_2$ then $I_e (J_1)\subseteq I_e (J_2)$.

\item One has that
\[
I_e \left(\sum_{i=1}^r J_i\right)=\sum_{i=1}^r I_e (J_i).
\]
Note that this fact implies that it is enough to know how to calculate $I_e (J)$ when $J$ is a principal ideal.

\item Let $g\in A$. If
\[
g=\sum_{\substack{b\in\mathcal{B}_e\\ 0\leq\norm{\mathbf{\alpha}}\leq p^e-1}}g_{\mathbf{\alpha}b}^{p^e}b \mathbf{x}^{\alpha}
\]
then $I_e (g)$ is the ideal of $A$ generated by all the $g_{\mathbf{\alpha}b}$'s.


\end{enumerate}

\end{prop}

Now, we are ready for expressing the condition of being $\phi$-fixed in computational terms. This is the main result of this section.
\begin{thm}\label{compatibles y fijos: como calcularlos}
Let $J\subseteq A$ be any ideal and let $\phi=u\Phi_e\in\Hom_A (F_*^e A,A)$. Then, the following statements hold.
\begin{enumerate}[(a)]

\item The image of $F_*^e J$ under $\phi$ is $I_e (uJ)$.

\item $J$ is $\phi$-compatible if and only if $I_e (uJ)\subseteq J$.

\item $J$ is $\phi$-fixed if and only if $I_e (uJ)=J$.

\end{enumerate}

\end{thm}

\begin{proof}
Parts (b) and (c) follow directly form part (a). So, it is enough to prove part (a).

Proposition \ref{propiedades del ideal raiz} implies that, in order to compute $I_e (uJ)$,
one may choose a set of generators $g_1,\ldots ,g_t$ of $F_*^e J$ and then compute $I_e (ug_1)+\ldots +I_e (ug_t)$.
Now, fix $1\leq i \leq t$ and write
\[
ug_i =\sum_{\substack{b\in\mathcal{B}_e\\ 0\leq\norm{\mathbf{\alpha}}\leq p^e-1}}r_{i\mathbf{\alpha}b}^{p^e}b \mathbf{x}^{\alpha}.
\]
Applying once more Proposition \ref{propiedades del ideal raiz},
it follows that $I_e (ug_i)$ is the ideal generated by all coefficients $r_{i\mathbf{\alpha}b}$ above.
But
\[
r_{i\mathbf{\alpha}b}=\Phi_e \left(F_*^e \left(b^{-1}x_1^{p^e-\alpha_1} \ldots x_d^{p^e-\alpha_d}\right)ug_i\right)\in\phi (F_*^e J),
\]
hence $I_e (ug_i)\subseteq\phi (F_*^e J)$ for any $1\leq i \leq t$ and
\[
I_e (uJ)\subseteq\phi (F_*^e J).
\]
Conversely, note that $\phi (y)=\Phi_e (uy)\in I_e (uJ)$ for any $y\in F_*^e J$, hence $\phi (F_*^e J)\subseteq I_e (uJ)$ and therefore we obtain the desired conclusion.
\end{proof}
Before going on, we want to single out in the below result an elementary characterization of compatible ideals because it will play some role later on in this paper (cf. proof of Lemma \ref{el referee me los toca bien}); it may be regarded as a consequence of Theorem \ref{compatibles y fijos: como calcularlos}.

\begin{cor}\label{vaya parida de equivalencia}
Let $\phi=u\Phi_e\in\Hom_A (F_*^e A,A)$, and let $J\subseteq A$ be an ideal. Then, $J$ is $\phi$-compatible if and only if $J\subseteq\left(J^{[p^e]}:_A u\right)$.
\end{cor}

\begin{proof}
According to Theorem \ref{compatibles y fijos: como calcularlos}, $J$ is $\phi$-compatible if and only if $J=I_e (uJ)$, which is equivalent to say that $J^{[p^e]}=I_e (uJ)^{[p^e]}$. This implies, since $I_e (uJ)^{[p^e]}\supseteq uJ$, that $J\subseteq\left(J^{[p^e]}:_A u\right)$.

Conversely, assume that $J\subseteq\left(J^{[p^e]}:_A u\right)$. This is equivalent to say that $uJ\subseteq J^{[p^e]}$, which implies that $I_e (uJ)\subseteq J$ by the definition of the $e$th root ideal.
\end{proof}

\begin{notation}
Henceforth, $S$ will denote the polynomial ring $\mathbb{K}[x_1,\ldots ,x_d]$
and $S_l$ will denote the $\mathbb{K}$-vector space generated by monomials $\mathbf{x}^{\mathbf{\alpha}}$ with
$\norm{\mathbf{\alpha}}\leq l$.
\end{notation}

The following result will guarantee that the algorithm we shall introduce later on (cf.\,Algorithm \ref{el algoritmo de Moty}) terminates after a finite number of steps.

\begin{prop}\label{acotamos la norma infinito de generadores}
The following statements hold.
\begin{enumerate}[(i)]

\item For any $y\in S$, the ideal $I_e (y)$ can be generated by elements $g\in S$ such that
\[
\norm{g}\leq\frac{\norm{y}}{p^e}.
\]

\item If $J$ is $u\Phi_e$-fixed then there exists a set of generators of $J$ such that if $g$ belongs to this set then
\[
\norm{g}\leq\frac{\norm{u}}{p^e-1}.
\]

\item If $J$ is $u\Phi_e$-fixed, then $\left(S_{D_e}\cap J\right)S=J$, where
\[
D_e:=\left\lceil\frac{\norm{u}}{p^e-1}\right\rceil.
\]

\end{enumerate}

\end{prop}

\begin{proof}
Since part (iii) follows immediately from part (ii), it is enough to show that parts (i) and (ii) hold.

We begin proving part (i). Indeed, we write
\[
y=\sum_{\substack{b\in\mathcal{B}_e\\ 0\leq\norm{\mathbf{\alpha}}\leq p^e-1}} y_{\mathbf{\alpha}b}^{p^e}b \mathbf{x}^{\mathbf{\alpha}}.
\]
In this way, for any $\mathbf{\alpha}$ and $b$ as above it follows that
\[
p^e \norm{y_{\mathbf{\alpha}b}}\leq\norm{y_{\mathbf{\alpha}b}^{p^e}}\leq\norm{y_{\mathbf{\alpha}b}^{p^e}\mathbf{x}^{\mathbf{\alpha}}}\leq\norm{y},
\]
whence part (i) holds.

Now, we prove part (ii). Let $M\geq 0$ be the minimal integer for which a set of generators of $J$ have norm at most $M$. Part (i) shows that $I_e (uJ)$ can be generated by polynomials with norm at most $\left(\norm{u}+M\right)/p^e$. In addition, as $I_e (uJ)=J$ we deduce, by the minimality of $M$, that $M\leq\left(\norm{u}+M\right)/p^e$ and therefore we conclude that $M\leq\norm{u}/(p^e-1)$, just what we finally wanted to check.
\end{proof}

\section{The algorithm through the hash operation}\label{the section of the algorithm}
The aim of this section is to describe a computational method to produce all the $u\Phi_e$-fixed ideals of $S$. As the reader will appreciate, our procedure is based on a new operation on ideals (cf.\,Definition \ref{la construccion hash}), which we hope to be of some interest in its own right.

We start with the following elementary statement, which we provide a proof for the sake of completeness. It may be regarded as an elementary consequence of Nakayama's Lemma.

\begin{lem}\label{una version del Nakayama}
Let $I\subseteq\mathfrak{m}$ be an ideal minimally generated by $s$ elements. Then, any ideal $J\subsetneq I$ is contained in some ideal $V$, where $\mathfrak{m}I\subseteq V\subseteq I$ and $\dim_{\mathbb{K}}I/V=1$.
\end{lem}

\begin{proof}
Nakayama's Lemma implies that there are $g_1,\ldots ,g_s\in S$ with $I=S g_1+\ldots +Sg_s$ such that $g_1,\ldots ,g_s \pmod{\mathfrak{m}I}$ is a basis of the $s$-dimensional $\mathbb{K}$-vector space $I/\mathfrak{m}I$. In this way, it follows that any ideal $J\subsetneq I$ is contained in some $V:=SW+\mathfrak{m} I$, where $W$ is a $(s-1)$-dimensional $\mathbb{K}$-vector subspace of $I/\mathfrak{m}I$. Moreover, we have to note as well that $\dim_{\mathbb{K}}I/V=1$.
\end{proof}

From now on, we shall assume that $u\in F_*^e S$ is fixed and set
\[
D_e:=\left\lceil\frac{\norm{u}}{p^e-1}\right\rceil.
\]
The following construction will be the crucial building block of our method.

\begin{defn}\label{la construccion hash}
Given any ideal $J\subseteq S$, we define the sequence of ideals
\[
J_0:=J,\quad J_{i+1}:=\left(J_i\cap\left(J_i^{[p^e]}:_S u\right)\cap I_e (uJ_i)\cap S_{D_e}\right)S,
\]
and set
\[
J^{\#_e}:=\bigcap_{i\geq 0}J_i.
\]
When $e=1$, we shall write $J^{\#}$ instead of $J^{\#_1}$ for the sake of brevity. Hereafter, we refer to this construction as the \emph{hash operation}.
\end{defn}

Now, we list in the below statement some elementary properties satisfied by the hash operation.

\begin{lem}\label{el referee me los toca bien}
Let $J,K\subseteq S$ be ideals of $S$. Then, the following assertions hold.

\begin{enumerate}[(a)]

\item If $J\subseteq K$, then $J^{\#_e}\subseteq K^{\#_e}$.

\item If $J$ is $u\Phi_e$-fixed, then $J=J^{\#_e}$.

\end{enumerate}
\end{lem}

\begin{proof}
First of all, we prove part (a); indeed, we show by increasing induction on $i\geq 0$ that $J_i\subseteq K_i$, where $J_i,K_i$ are as in Definition \ref{la construccion hash}. This is clearly true for $i=0$.

Now, we assume that $i\geq 0$ and that $J_i\subseteq K_i$. Since $uJ_i\subseteq uK_i$, it follows from part (a) of Proposition \ref{propiedades del ideal raiz} that $I_e (uJ_i)\subseteq I_e (uK_i)$. Moreover, since $J_i^{[p^e]}\subseteq K_i^{[p^e]}$ it also follows that $\left(J_i^{[p^e]}:_S u\right)\subseteq\left(K_i^{[p^e]}:_S u\right)$. Summing up, one has that
\begin{align*}
J_{i+1}& = \left(J_i\cap\left(J_i^{[p^e]}:_S u\right)\cap I_e (uJ_i)\cap S_{D_e}\right)S\\ & \subseteq\left(K_i\cap\left(K_i^{[p^e]}:_S u\right)\cap I_e (uK_i)\cap S_{D_e}\right)S=K_{i+1},
\end{align*}
whence part (a) holds.

In this way, it only remains to prove that part (b) is also true; indeed, suppose now that $J$ is $u\Phi_e$-fixed. We shall show by increasing induction on $i\geq 0$ that $J=J_i$ for all $i\geq 0$, where $J_i$ is as in Definition \ref{la construccion hash}. This is clearly true for $i=0$.

Now, we assume that $i\geq 0$ and that $J=J_i$. Since $J$ is $u\Phi_e$-fixed, one has that $J=I_e (uJ)=I_e (uJ_i)$; moreover, it follows from Corollary \ref{vaya parida de equivalencia} and part (iii) of Proposition \ref{acotamos la norma infinito de generadores} respectively that $J\subseteq\left(J^{[p^e]}:_S u\right)=\left(J_i^{[p^e]}:_S u\right)$ and $\left(S_{D_e}\cap J_i\right)S=\left(S_{D_e}\cap J\right)S=J$, whence
\[
J_{i+1} = \left(J_i\cap\left(J_i^{[p^e]}:_S u\right)\cap I_e (uJ_i)\cap S_{D_e}\right)S=J,
\]
just what we finally wanted to check.
\end{proof}
Next result may be regarded as an elementary consequence of Lemma \ref{el referee me los toca bien}.

\begin{cor}\label{los fijos quedan atrapados por la operacion hash}
For any ideal $J\subseteq S$, $J^{\#_e}$ contains all the $u\Phi_e$-fixed ideals which are contained in $J$.
\end{cor}

\begin{proof}
Let $I\subseteq J$ be any $u\Phi_e$-fixed ideal. Applying Lemma \ref{el referee me los toca bien} it follows that $I=I^{\#_e}\subseteq J^{\#_e}$.
\end{proof}

The introduction of the hash operation is motivated by the following result.

\begin{thm}\label{locura de jueves}
$J^{\#_e}$ is the greatest $u\Phi_e$-fixed ideal contained in $J$.
\end{thm}

\begin{proof}
We only need to check that $J^{\#_e}$ is $u\Phi_e$-fixed, because the other assertions of the Theorem are clear regarding Corollary \ref{los fijos quedan atrapados por la operacion hash}.

First of all, we prove that $J^{\#_e}$ is $u\Phi_e$-compatible; indeed, set
\[
n:=\min\{i\in\mathbb{N}\mid\ J_i=J_{i+1}\},
\]
where $J_i$ is as in Definition \ref{la construccion hash}. We have to point out that $J^{\#_e}=J_n$, because the decreasing sequence of ideals $\{J_i\}_{i\in\mathbb{N}}$ stabilizes rigidly (this is due to the fact that, at each step, we are intersecting with the finite dimensional $\K$-vector space $S_{D_e}$). Therefore, one has that
\begin{align*}
J^{\#_e}=J_n=J_{n+1}=& \left(J_n\cap\left(J_n^{[p^e]}:_S u\right)\cap I_e (uJ_n)\cap S_{D_e}\right)S\\ &\subseteq\left(J_n^{[p^e]}:_S u\right)=\left(\left(J^{\#_e}\right)^{[p^e]}:_S u\right).
\end{align*}
From the previous displayed upper inclusion it follows, by means of Corollary \ref{vaya parida de equivalencia}, that $J^{\#_e}$ is $u\Phi_e$-compatible, whence $I_e \left(uJ^{\#_e}\right)\subseteq J^{\#_e}$. A similar argument shows that
\[
J^{\#_e}=J_n=J_{n+1}\subseteq I_e (uJ_n)=I_e \left(uJ^{\#_e}\right)
\]
and therefore we can ensure that $J^{\#_e}$ is $u\Phi_e$-fixed, just what we wanted to show.
\end{proof}




\subsection{The statement of the algorithm}

Now, we introduce our promised algorithm. More precisely, the next result is a recursive procedure for producing all the $u\Phi_e$-fixed ideals of $S$.

This is the main result of this paper.

\begin{thm}\label{Moty nos da un algoritmo}
Let $I\subseteq\mathfrak{m}$. The set $\FP_e (I)$ of all $u\Phi_e$-fixed ideals contained in $I$ is given recursively as $\FP_e (\langle 0\rangle)=\{\langle 0\rangle\}$ and, for $I\neq\langle 0\rangle$, defined as the union of $\{I^{\#_e}\}$ (whenever $I^{\#_e}$ is $u\Phi_e$-fixed) and
\[
\bigcup\left\{\FP_e (V)\mid\quad\mathfrak{m}I^{\#_e}\subseteq V\subseteq I^{\#_e},\quad\dim_{\mathbb{K}} I^{\#_e}/V=1\right\}.
\]
Moreover, if $\K$ is finite then this recursion is finite in the sense that the resulting execution tree is finite.
\end{thm}

\begin{proof}
Firstly, we show that if $J\subseteq I$ is $u\Phi_e$-fixed then $J\in\FP_e (I)$. We shall proceed by increasing induction on $t:=\dim_{\mathbb{K}} (I\cap S_{D_e})$; indeed, if $t=0$ then $J\subseteq I^{\#_e}=\langle 0\rangle$ and therefore $J\in\{\langle 0\rangle\}=\FP_e (I)$.

Now, let $J\subseteq I$ be such that $t\geq 1$. If $J= I^{\#_e}$ then we are done by Corollary \ref{los fijos quedan atrapados por la operacion hash}. Thus, we assume that $J\subsetneq I^{\#_e}$. Since $I\subseteq\mathfrak{m}$, Lemma \ref{una version del Nakayama} says us that we can find an ideal $\mathfrak{m}I^{\#_e}\subseteq V\subsetneq  I^{\#_e}$ such that $\dim_{\mathbb{K}} I^{\#_e}/V=1$ and $J\subseteq V$. Furthermore, by construction, $I^{\#_e}$ can be generated by elements in $S_{D_e}$, hence $V\cap S_{D_e}\subsetneq I^{\#_e}\cap S_{D_e}$ and therefore the induction hypothesis implies that $J\in\FP_e (V)\subseteq\FP_e (I)$.

Finally, we have to point out that our foregoing inductive argument shows that the chains of $V^{\#_e}$'s produced in this recursion have length at most $\dim_{\mathbb{K}}S_{D_e}$, hence the second statement follows too.
\end{proof}

In this way, we can turn Theorem \ref{Moty nos da un algoritmo} into an effective method to calculate all the $u\Phi_e$-fixed ideals of any polynomial ring having a finite field as field of coefficients as follows.

\begin{algo}\label{el algoritmo de Moty}
Let $\K$ be a finite field of prime characteristic, set $S:=\K [x_1,\ldots ,x_d]$ and let $u\in S$. These data act as the input of the procedure. Moreover, we initialize $I$ as the whole ring $S$ and $L$ as the empty list $\{\}$.
\begin{enumerate}[(i)]

\item Compute $I^{\#_e}$. Assign to $I$ the value of $I^{\#_e}$.

\item If $I$ is not in the list $L$, then add it.


\item If $I=0$, then stop and output the list $L$.

\item If $I\neq 0$ but principal, assign to $I$ the value of $\mathfrak{m}I$ and come back to step (i).

\item If $I\neq 0$ and not principal, then compute
\[
\left\{V\quad\text{ideal}\mid\quad\mathfrak{m}I\subseteq V\subseteq I,\quad\dim_{\mathbb{K}} I/V=1\right\}.
\]
For each element $V$ of the previous set, come back to step (i).

\end{enumerate}
At the end of this method, the list $L$ contain all the $u\Phi_e$-fixed ideals of $S$ which are contained in $\mathfrak{m}$.
\end{algo}

\begin{rem}
The reader should notice that step (v) of the previous method is the only reason for which we have to assume that our coefficient field $\K$ is finite; otherwise, the set $\left\{V\quad\text{ideal}\mid\quad\mathfrak{m}I\subseteq V\subseteq I,\quad\dim_{\mathbb{K}} I/V=1\right\}$ is not finite.
\end{rem}

\begin{rem}
It is worth mentioning the following facts about the complexity of Algorithm \ref{el algoritmo de Moty}. On one hand, as pointed out during the proof of Theorem \ref{Moty nos da un algoritmo}, each chain of fixed ideals produced by our method has length at most $\dim_{\K} (S_{D_e})$; however, we have no control about how many chains of fixed ideals can appear. On the other hand, given an ideal $J$ of $S$, the calculation of the $e$-th root $I_e (J)$ is linear not only in $p^e$, but also in the number of generators of $J$; moreover, if $\K$ is field of $q$ elements ($q=p^f$ for some $f\geq 1$), then the cardinality of
\[
\left\{V\quad\text{ideal}\mid\quad\mathfrak{m}I\subseteq V\subseteq I,\quad\dim_{\mathbb{K}} I/V=1\right\}.
\]
is exactly $1+q+q^2+\ldots+q^{t-1}$, where $t$ denotes the number of generators of $J$; regardless, we have no control about what dimensions of $J/\mathfrak{m}J$ one finds during the recursion.
\end{rem}








We end this section with the following result, which may be regarded as an elementary consequence of the very definition of the hash operation.

\begin{cor}
Let $\K$ be any $F$-finite field of prime characteristic $p$, set $S:=\K [x_1,\ldots ,x_d]$, and let $u\in S$. Then, the ideal $\langle u\rangle$ is a minimal $\phi$-fixed ideal, where $\phi:=u^{p^e-1}\Phi_e$.
\end{cor}

\begin{proof}
We have to check that $\langle u\rangle$ is a minimal non-zero $\phi$-fixed ideal of $S$. Firstly, we show that $\langle u\rangle$ is $\phi$-fixed; indeed,
\[
I_e (u^{p^e-1}\cdot\langle u\rangle)=I_e (u^{p^e})=\langle u\rangle ,
\]
whence $\langle u\rangle$ is $\phi$-fixed. So, it only remains to prove that $\langle u\rangle$ is a minimal $\phi$-fixed ideal of $S$. First of all, notice that
\[
D_e= \left\lceil\frac{\norm{u^{p^e-1}}}{p^e-1}\right\rceil =\norm{u}.
\]
On the other hand, as $\langle u\rangle$ is principal it follows, combining Lemma \ref{una version del Nakayama} and Corollary \ref{los fijos quedan atrapados por la operacion hash}, that if $I\subsetneq\langle u\rangle$ is $\phi$-fixed, then $I\subseteq\left(\langle u\rangle\cdot\mathfrak{m}\right)^{\#_e}$. This implies that any element $g\in\langle u\rangle\cdot\mathfrak{m}$ which forms part of a system of generators for $\langle u\rangle\cdot\mathfrak{m}$ is such that
\[
\norm{g}\geq\norm{u}+1>\norm{u}.
\]
From this strict lower inequality it follows that $S_{D_e}\cap\left(\langle u\rangle\right)$ is empty, whence $\left(\langle u\rangle\cdot\mathfrak{m}\right)^{\#_e}=0$ and therefore $I=0$, just what we finally wanted to check.
\end{proof}

\section{Examples}\label{examples section}

The goal of this section is to present some interesting calculations which were carried out with an implementation of the algorithm presented in this manuscript. Macaulay2 (cf.\,\cite{M2}) has been used extensively both in constructing and exploring examples, as well as implementing the procedure described herein.

Firstly, we include an example where we develop the algorithm step by step for the convenience of the reader.

\begin{ex}
We consider the ring $S:=\mathbb{F}_2 [x,y]$ and set $u:=xy$. We compute $\FP_1 (S)$.
\begin{enumerate}[(a)]

\item Start with $I=S=I^{\#}$. As $I_1 (uI)=I$ add $S$ to the list $\FP_1 (S)$.

\item As $I$ is principal, go on with $I=\mathfrak{m}=I^{\#}$. Since $I_1 (uI)=I$ add $\mathfrak{m}$ to the list $\FP_1 (S)$. Moreover, we have to note that
\[
\left\{\mathfrak{m}^2\subseteq V\subseteq\mathfrak{m}\mid\quad\dim_{\mathbb{F}_2} \mathfrak{m}/V=1\right\}=\left\{\langle x,y^2\rangle, \langle y,x^2\rangle, \langle x^2 ,xy, x+y\rangle\right\}.
\]
We have to emphasize that in the calculation of this set is when we are using that we are working with characteristic two.

Thus, we need to compute the following sets of fixed ideals:
\[
\FP_1 (\langle x^2 ,xy, x+y\rangle),\FP_1 (\langle x,y^2\rangle)\text{ and }\FP_1 (\langle y,x^2\rangle).
\]
As $\langle x^2 ,xy, x+y\rangle^{\#}=\langle xy\rangle$ and $I_1 (u\langle xy\rangle)=\langle xy\rangle$ add $\langle xy\rangle$ to the list $\FP_1 (\langle x^2 ,xy, x+y\rangle)$. Moreover, as $\langle xy\rangle$ is principal go on with $\langle x^2 y,xy^2\rangle$. Nevertheless, since $\langle x^2 y,xy^2\rangle^{\#}=\langle 0\rangle$ we deduce that
\[
\FP_1 (\langle x^2 ,xy, x+y\rangle)=\{\langle xy\rangle ,\langle 0\rangle\}.
\]

On the other hand, since $\langle x,y^2\rangle^{\#}=\langle x\rangle$ and $I_1 (u\langle x\rangle)=\langle x\rangle$ add $\langle x\rangle$ to the list $\FP_1 (\langle x,y^2\rangle)$. In addition, as $\langle x\rangle$ is principal go on with $\langle x^2 ,xy\rangle$. However, since $\langle x^2, xy\rangle^{\#}=\langle xy\rangle$ we can use the foregoing calculations and therefore we conclude that $\FP_1 (\langle x,y^2\rangle)=\{\langle x\rangle ,\langle xy\rangle ,\langle 0\rangle\}$.

A similar computation shows that $\FP_1 (\langle y,x^2\rangle)=\{\langle y\rangle, \langle xy\rangle, \langle 0\rangle\}$.

\end{enumerate}
In this way, it follows that $\FP_1 (S)=\left\{\mathbb{F}_2 [x,y], \langle x,y\rangle, \langle x\rangle ,\langle y\rangle, \langle xy\rangle, \langle 0\rangle\right\}$.
\end{ex}

Secondly, we include a more involved example which was studied in greater detail in \cite[Section 9]{Katzman2008}.

\begin{ex}
Consider the matrix of variables
\[
A:=\begin{pmatrix} x_1& x_2& x_2& x_5\\ x_4& x_4& x_3& x_1\end{pmatrix}
\]
and set $S:=\mathbb{F}_2 [x_1,x_2,x_3,x_4,x_5]$. Furthermore, for any $1\leq i<j\leq 4$ we denote by $M_{ij}$ the minor of $A$ of size $2$ obtained from columns $i$ and $j$. In addition, set
\[
u:=x_1^3x_2x_3+x_1^3x_2x_4+x_1^2x_3x_4x_5+x_1x_2x_3x_4x_5+x_1x_2x_4^2x_5+x_2^2x_4^2x_5+x_3x_4^2x_5^2+x_4^3x_5^2.
\]
Our procedure produces the following $84$ proper $\phi$-fixed ideals of $S$, where $\phi:=u\Phi_1$.

\begin{enumerate}[(i)]

\item One prime ideal generated by five elements; namely, the ideal $\mathfrak{m}$ generated by all the variables of $S$.

\item Four prime ideals generated by four elements; namely,
\[
\langle x_1,x_2,x_3,x_4\rangle,\langle x_1,x_2,x_4,x_5\rangle,\langle x_1,x_3,x_4,x_5\rangle,\langle x_1,x_2,x_3+x_4,x_5\rangle.
\]

\item Five prime ideals generated by three elements; namely,
\[
\langle x_1,x_2,x_5\rangle, \langle x_1,x_3,x_4\rangle,\langle x_1,x_2,x_4\rangle,\langle x_1,x_4,x_5\rangle,\langle x_1+x_2,x_3+x_4,x_2^2+x_4x_5\rangle.
\]

\item Two prime ideals generated by two elements; namely, $\langle x_1,x_4\rangle$ and $\langle x_1+x_2,x_2^2+x_4x_5\rangle$. The reader should notice that
\[
\langle x_1+x_2,x_2^2+x_4x_5\rangle =\langle x_1+x_2,x_1^2+x_4x_5\rangle
\]
because of we are working on characteristic two.

\item One prime ideal generated by just one element; namely, the ideal $\langle u\rangle$.

\item Twenty-nine ideals which contains in their set of minimal generators some $M_{ij}$ for some $1\leq i<j\leq 4$.

\item The remainder fourty-two ideals define arrangements of linear varieties. Among these $42$ ideals, there is one distinguished element; namely, the ideal $\langle x_1, x_2, x_3 +x_4, x_4x_5\rangle$. In \cite[Section 9]{Katzman2008} it was shown that this ideal is the parameter test ideal of the quotient ring $S/I$, where $I$ is the ideal of $S$ generated by the $2\times 2$ minors of $A$.

\end{enumerate}
The reader should notice that, in this case, the set of $\phi$-fixed ideals equals the set of $\phi$-compatible ideals; indeed, this is due to the fact that, in this case, the map $\phi=u\Phi_1$ is a Frobenius splitting. In particular, we recover the thirteen non-zero $\phi$-compatible primes obtained by M.\,Katzman and K.\,Schwede in \cite[Example 7.2]{KatzmanSchwede2012}.
\end{ex}

Thirdly, we include an example where the characteristic of our ground field is greater than two.

\begin{ex}
Let $S:=\mathbb{F}_5 [x,y,z]$, and  $u=\left(x^4+y^4+z^4\right)^4$. The aim of this example is to compute, using our algorithm, all the $\phi$-fixed ideals of $S$, where $\phi :=u \Phi_1$. Our method produces the following sixty-five non-zero $\phi$-fixed ideals. 

\begin{enumerate}[(i)]

\item The ideal $\mathfrak{m}:=\langle x,y,z\rangle\subseteq S$, its square $\mathfrak{m}^2$ and the principal ideal $\langle u\rangle$.

\item Thirty-one ideals of the form $\mathfrak{m}^2+H$, where $H$ is an ideal of $S$ generated by a single linear form. 

\item Thirty-one ideals of the form $\mathfrak{m}^2+G$, where $G$ is an ideal of $S$ generated by two linear forms. 

\end{enumerate}
It is worth noting that this specific calculation is also interesting because it provides an example where our method provides more information than the procedures worked out in \cite{KatzmanSchwede2012}; indeed, if one uses \cite{FsplittingM2} here, then one only gets the ideals $\langle u\rangle$ and $\mathfrak{m}$. As we have explained in the Introduction, the reader should remember that, whereas our algorithm produces all the $\phi$-fixed ideals, the procedures described in \cite{KatzmanSchwede2012} describes algorithmically the radical $\phi$-compatible ideals.
\end{ex}

Finally, we conclude this paper with the following example, which was studied in greater detail in \cite[Section 2]{Katzman2010}.

\begin{ex}
We fix the $2\times 3$ matrix of indeterminates
\[
\begin{pmatrix} x_1& x_2& x_3\\ y_1& y_2& y_3\end{pmatrix}
\]
and we let $S$ to be the polynomial ring $\mathbb{F}_2 [x_1,x_2,x_3, y_1,y_2,y_3]$. Moreover, for any $1\leq i<j\leq 3$ $\Delta_{ij}$ will stand for the $2\times 2$ minor obtained from columns $i$ and $j$. In this way, taking into account this notation, we set
\[
u:=\Delta_{12}\Delta_{13}=(x_1 y_2-x_2 y_1)(x_1 y_3-x_3 y_1).
\]
Our procedure produces the following seven proper $\phi$-fixed ideals, where $\phi:=u\Phi_1$; namely,
\[
\langle x_1,y_1,\Delta_{23}\rangle , \langle x_1,y_1\rangle , \langle \Delta_{12},\Delta_{13},\Delta_{23}\rangle , \langle\Delta_{12},\Delta_{13}\rangle , \langle\Delta_{12}\rangle , \langle\Delta_{13}\rangle , \langle\Delta_{12}\Delta_{13}\rangle .
\]
In particular, we obtain the following five proper $\phi$-fixed prime ideals:
\[
\langle\Delta_{12}\rangle , \langle\Delta_{13}\rangle , \langle x_1,y_1\rangle , \langle x_1,y_1,\Delta_{23}\rangle , \langle \Delta_{12},\Delta_{13},\Delta_{23}\rangle .
\]
Such list of $\phi$-fixed prime ideals turns out to be the complete list of proper $\phi$-compatible prime ideals, as the reader can check using \cite{FsplittingM2}.
\end{ex}

\bibliographystyle{plain}
\bibliography{sepshh}

\end{document}